\documentclass[12pt,reqno]{amsart}
\usepackage{amsfonts}
\usepackage{amscd}
\usepackage{graphicx}
\def\frk{\frak}

\def\Phi{{\frk n}}
\def\Phi{{\frk N}}

\def\opn#1#2{\def#1{\operatorname{#2}}}
\opn\chara{char} \opn\length{\ell} \opn\pd{pd} \opn\rk{rk}
\opn\projdim{proj\,dim} \opn\injdim{inj\,dim} \opn\rank{rank}
\opn\depth{depth} \opn\grade{grade} \opn\height{height}
\opn\embdim{emb\,dim} \opn\codim{codim}

\opn\Tr{Tr} \opn\bigrank{big\,rank}
\opn\superheight{superheight}\opn\lcm{lcm}
\opn\trdeg{tr\,deg}\opn\reg{reg} \opn\lreg{lreg} \opn\ini{in} \opn\lpd{lpd}
\opn\size{size}
\opn\div{div} \opn\Div{Div} \opn\cl{cl} \opn\Cl{Cl}
\opn\Spec{Spec} \opn\Supp{Supp} \opn\supp{supp} \opn\Sing{Sing}
\opn\Ass{Ass} \opn\Min{Min}
\opn\Ann{Ann} \opn\Rad{Rad} \opn\Soc{Soc}
\opn\Im{Im} \opn\Ker{Ker} \opn\Coker{Coker} \opn\Am{Am}
\opn\Hom{Hom} \opn\Tor{Tor} \opn\Ext{Ext} \opn\End{End}
\opn\Aut{Aut} \opn\id{id}

\opn\nat{nat}
\opn\pff{pf}\opn\Pf{Pf} \opn\GL{GL} \opn\SL{SL} \opn\mod{mod} \opn\ord{ord}
\opn\Gin{Gin} \opn\Hilb{Hilb}
\opn\aff{aff} \opn\con{conv} \opn\relint{relint} \opn\st{st}
\opn\lk{lk} \opn\cn{cn} \opn\core{core} \opn\vol{vol}
\opn\link{link} \opn\star{star}
\opn\gr{gr}

\def\pot#1#2{#1[\kern-0.28ex[#2]\kern-0.28ex]}

\opn\dirlim{\underrightarrow{\lim}}
\opn\inivlim{\underleftarrow{\lim}}

\def\Implies{\ifmmode\Longrightarrow \else
        \unskip${}\Longrightarrow{}$\ignorespaces\fi}
\def\implies{\ifmmode\Rightarrow \else
        \unskip${}\Rightarrow{}$\ignorespaces\fi}
\def\iff{\ifmmode\Longleftrightarrow \else
        \unskip${}\Longleftrightarrow{}$\ignorespaces\fi}

\let\:=\colon
\newtheorem{Theorem}{Theorem}[section]
\newtheorem{Lemma}[Theorem]{Lemma}

\newtheorem{Proposition}[Theorem]{Proposition}

\newtheorem{Problem}[Theorem]{Problem}

\newtheorem{Example}[Theorem]{Example}

\newtheorem{Definition}[Theorem]{Definition}

\let\epsilon\varepsilon
\let\phi=\varphi
\let\kappa=\varkappa
\def\qed{\ifhmode\textqed\fi
      \ifmmode\ifinner\quad\qedsymbol\else\dispqed\fi\fi}
\def\textqed{\unskip\nobreak\penalty50
       \hskip2em\hbox{}\nobreak\hfil\qedsymbol
       \parfillskip=0pt \finalhyphendemerits=0}
\def\dispqed{\rlap{\qquad\qedsymbol}}

\opn\dis{dis}
\def\pnt{{\raise0.5mm\hbox{\large\bf.}}}

\opn\Lex{Lex}

\begin{document}
\title{Fixed point approximation of suzuki generalized nonexpansive
mapping via new faster iteration process}
\author{Nawab Hussain}
\address{Department of Mathematics, King Abdulaziz University, P.O. Box 80203,
Jeddah, 21589, Saudi Arabia.}
\email{nhusain@kau.edu.sa}
\author{Kifayat Ullah}
\address{Department of Mathematics, International Islamic University Islamabad, Pakistan.}
\email{kifayatmath@yahoo.com}
\author{Muhammad Arshad}
\address{Department of Mathematics, International Islamic University Islamabad, Pakistan.}
\email{marshadzia@iiu.edu.pk}

\begin{abstract}  In this paper we propose a new iteration
process, called the $K$ iteration process, for approximation of fixed
points. We show that our iteration process is faster than the existing
leading iteration processes like Picard-S iteration process, Thakur New
iteration process and \ Vatan Two-step iteration process for contraction
mappings. We support our analytic proof by a numerical example. Stability of 
$K$ iteration process and data dependence result for contraction mappings by
employing $K$ iteration process is also discussed. Finally we prove some
weak and strong convergence theorems for the Suzuki generalized nonexpansive
mappings in the setting of uniformly convex Banach space. Our results are
extension, improvement and generalization of many known results in the
literature of fixed point theory.
\end{abstract}
\maketitle

\section{introduction}

Fixed point theory takes a large amount of literature, since it provides
useful tools to solve many problems that have applications in different
fields like engineering, economics, chemistry and game theory etc. However,
once the existence of a fixed point of some mapping is established, then to
find the value of that fixed point is not an easy task, that is why we use
iterative processes for computing them. By time, many iterative processes
have been developed and it is impossible to cover them all. The well-known
Banach contraction theorem use Picard iteration process for approximation of
fixed point. Some of the other well-known iterative processes are Mann \cite%
{10}, Ishikawa \cite{8}, Agarwal \cite{4}, Noor \cite{11}, Abbas \cite{1},
SP \cite{14a}, S$^{\ast }$ \cite{8aa}, CR \cite{4a}, Normal-S \cite{16a},
Picard Mann \cite{8c}, Picard-S \cite{7a}, Thakur New \cite{19a}, Vatan
Two-step \cite{8a} and so on.

Two qualities "Fastness" and "Stability" play important role for an
iteration process to be preferred on another iteration process. In \cite{15}%
, Rhoades mentioned that the Mann iteration process for decreasing function
converge faster than the Ishikawa iteration process and for increasing
function the Ishikawa iteration process is better than the Mann process.
Also the Mann iteration process appears to be independent of the initial
guess (see also \cite{16}). In \cite{4}, the authors claimed that Agarwal
iteration process converge at a rate same as that of the Picard iteration
process and faster than the Mann iteration process for contraction mappings.
In \cite{1}, the authors claimed that Abbas iteration process converge
faster than Agarwal iteration process. In \cite{4a}, the authors claimed
that CR iteration process is equivalent to and faster than Picard, Mann,
Ishikawa, Agarwal, Noor and SP iterative processes for quasi-contractive
operators in Banach spaces. Also in \cite{8b}\ the authors proved that CR
iterative process converge faster than the S$^{\ast }$ iterative process for
the class of contraction mappings. In \cite{7a}, authors claimed that
Picard-S iteration process is converge faster than all Picard, Mann,
Ishikawa, Noor, SP, CR, Agarwal, S$^{\ast },$ Abbas and Normal-S for
contraction mappings. In \cite{19a}, the authors proved with the help of
numerical example that Thakur New iteration process converge faster than
Picard, Mann, Ishikawa, Agarwal, Noor and Abbas iteration processes for the
class of Suzuki generalized nonexpansive mappings. Similarly in \cite{8a},
the authors proved that Vatan Two-step iteration process is faster than
Picard-S, CR, SP and Picard-Mann iteration processes for weak contraction
mappings. For Jungck-type iterative processes and their speed comparison see 
\cite{AKH, HKK, HJ, 8d}.

Motivated by above, in this paper, we introduce a new iteration process and
then prove analytically that our process is stable. Then we prove that $K$
iteration process converges faster than Picard-S iteration process and hence
faster than all Picard, Mann, Ishikawa, Noor, SP, CR, S, S$^{\ast },$ Abbas,
Normal-S and Two-step Mann iteration processes for contraction mappings.
Numerically we compare the convergence of the $K$ iteration process with the
three most leading iteration processes in the existing literature for
contraction mapping. The data dependence result for fixed point of
contraction mappings with the help of the $K$ iteration process is proved.
Finally we prove some weak and strong convergence theorems for Suzuki
generalized nonexpansive mappings$,$ which is the generalization of
nonexpansive as well as contraction mappings, in the setting of uniformly
convex Banach spaces.

\section{Preliminaries}

We now recall some definitions, propositions and lemmas to be used in the
next two sections.

A Banach space $X$ is called uniformly convex \cite{7} if for each $%
\varepsilon \in $ $(0,2]$ there is a $\delta $ $>0$ such that for $x,y\in X$,%
\begin{equation*}
\left. 
\begin{array}{c}
\left \Vert x\right \Vert \leq 1, \\ 
\left \Vert y\right \Vert \leq 1, \\ 
\left \Vert x-y\right \Vert >\varepsilon%
\end{array}%
\right \} \Longrightarrow \left \Vert \frac{x+y}{2}\right \Vert \leq \delta .
\end{equation*}

A Banach space $X$ is said to satisfy the Opial property \cite{12} if for
each sequence $\{x_{n}\}$ in $X,$ converging weakly to $x\in X,$ we have

\begin{equation*}
\underset{n\rightarrow \infty }{\lim \sup }\left \Vert x_{n}-x\right \Vert <%
\underset{n\rightarrow \infty }{\lim \sup }\left \Vert x_{n}-y\right \Vert ,
\end{equation*}
for all $y\in X$ such that $y\neq x$.

A point $p$ is called fixed point of a mapping $T$ if $T(p)=p$, and $F(T)$
represents the set of all fixed points of mapping $T.$ Let $C$ be a nonempty
subset of a Banach space $X.$ A mapping $T:C\rightarrow C$ is called
contraction if there exists $\theta \in (0,1)$ such that $\left \Vert
Tx-Ty\right \Vert \leq \theta \left \Vert x-y\right \Vert ,$ for all $x,y\in
C.$ A mapping $T:C\rightarrow C$ is called nonexpansive if $\left \Vert
Tx-Ty\right \Vert \leq \left \Vert x-y\right \Vert $ for all $x,y\in C,$ and
quasi-nonexpansive if for all $x\in C$ and $p\in F(T),$ we have $\left \Vert
Tx-p\right \Vert \leq \left \Vert x-p\right \Vert $. In 2008, Suzuki \cite%
{19} introduced the concept of generalized nonexpansive mappings which is a
condition on mappings called condition $(C)$. A mapping $T:C\rightarrow C$
is said to satisfy condition $(C)$ if for all $x,y\in C$, we have%
\begin{equation*}
\frac{1}{2}\left \Vert x-Tx\right \Vert \leq \left \Vert x-y\right \Vert 
\text{ implies }\left \Vert Tx-Ty\right \Vert \leq \left \Vert x-y\right
\Vert .
\end{equation*}

Suzuki \cite{19} showed that the mapping satisfying condition $(C)$ is
weaker than nonexpansiveness and stronger than quasi nonexpansiveness. He
also obtained fixed point theorems and convergence theorems for such
mappings. In 2011, Phuengrattana \cite{14} proved convergence theorems for
mappings satisfying condition $(C)$ using the Ishikawa iteration in
uniformly convex Banach spaces and $CAT(0)$ spaces. Recently, fixed point
theorems for mapping satisfying condition $(C)$ have been studied by a
number of authors see e.g.\cite{19a} and references therein.

We now list some properties of mapping that satisfy condition $(C)$.

\begin{Proposition}
\label{p1}Let $C$ be a nonempty subset of a Banach space $X$ and $%
T:C\rightarrow C$ be any mapping. Then

(i) \cite[Proposition 1]{19} If $T$ is nonexpansive then T satisfies
condition $(C)$.

(ii) \cite[Proposition 2]{19} If $T$ satisfies condition $(C)$ and has a
fixed point, then $T$ is a quasi-nonexpansive mapping.

(iii) \cite[Lemma 7]{19} If $T$ satisfies condition $(C)$, then $\left \Vert
x-Ty\right \Vert \leq 3\left \Vert Tx-x\right \Vert +\left \Vert
x-y\right
\Vert $ for all $x,y\in C$.
\end{Proposition}

\begin{Lemma}
\label{l1}\cite[Proposition 3]{19} Let $T$ be a mapping on a subset $C$ of a
Banach space $X$ with the Opial property. Assume that $T$ satisfies
condition $(C)$. If $\{x_{n}\}$ converges weakly to $z$ and $%
\lim_{n\rightarrow \infty }\left \Vert Tx_{n}-x_{n}\right \Vert =0$, then $%
Tz=z $. That is, $I-T$ is demiclosed at zero.
\end{Lemma}

\begin{Lemma}
\label{l2}\cite[Theorem 5]{19} Let $C$ be a weakly compact convex subset of
a uniformly convex Banach space $X$. Let $T$ be a mapping on $C$. Assume
that $T$ satisfies condition $(C)$. Then $T$ has a fixed point.
\end{Lemma}

\begin{Lemma}
\label{l3}\cite[Lemma 1.3]{17} Suppose that $X$ is a uniformly convex Banach
space and $\{t_{n}\}$ be any real sequence such that $0<p\leq t_{n}\leq q<1$
for all $n\geq 1$. Let $\{x_{n}\}$ and $\{y_{n}\}$ be any two sequences of $%
X $ such that $\lim \sup_{n\rightarrow \infty }\left \Vert x_{n}\right \Vert
\leq r$, $\lim \sup_{n\rightarrow \infty }\left \Vert y_{n}\right \Vert \leq
r$ and $\lim \sup_{n\rightarrow \infty }\left \Vert
t_{n}x_{n}+(1-t_{n})y_{n}\right \Vert $ $=r$ hold for some $r\geq 0$. Then $%
\lim {}_{n\rightarrow \infty }$ $\left \Vert x_{n}-y_{n}\right \Vert =0$.
\end{Lemma}

Let $C$ be a nonempty closed convex subset of a Banach space $X$, and let $%
\{x_{n}\}$ be a bounded sequence in $X$. For $x\in X$, we set%
\begin{equation*}
r(x,\{x_{n}\})=\lim \sup_{n\rightarrow \infty }\left \Vert x_{n}-x\right
\Vert .
\end{equation*}

The asymptotic radius of $\{x_{n}\}$ relative to $C$ is given by

\begin{equation*}
r(C,\{x_{n}\})=\inf \{r(x,\{x_{n}\}):x\in C\},
\end{equation*}

and the asymptotic center of $\{x_{n}\}$ relative to $C$ is the set

\begin{equation*}
A(C,\{x_{n}\})=\{x\in C:r(x,\{x_{n}\})=r(C,\{x_{n}\})\}.
\end{equation*}

It is known that, in a uniformly convex Banach space, $A(C,\{x_{n}\})$
consists of exactly one point.

\begin{Definition}
\cite{4aa} Let $\{u_{n}\}_{n=0}^{\infty }$ and $\{v_{n}\}_{n=0}^{\infty }$
be two fixed point iteration procedure sequences that converge to the same
fixed point $p$ and $\left \Vert u_{n}-p\right \Vert \leq a_{n}$ and $%
\left
\Vert v_{n}-p\right \Vert \leq b_{n}$ for all $n\geq 0$. If the
sequences $\{a_{n}\}_{n=0}^{\infty }$ and $\{b_{n}\}_{n=0}^{\infty }$
converge to $a$ and $b$, respectively, and $\lim_{n\rightarrow \infty }\frac{%
\left \Vert a_{n}-a\right \Vert }{\left \Vert b_{n}-b\right \Vert }=0$, then
we say that $\{u_{n}\}_{n=0}^{\infty }$ converge faster than $%
\{v_{n}\}_{n=0}^{\infty }$ to $p$.
\end{Definition}

\begin{Definition}
\cite{7b} Let $\{t_{n}\}_{n=0}^{\infty }$ be an arbitrary sequence in $C$.
Then, an iteration procedure $x_{n+1}=f(T,x_{n})$ converging to fixed point $%
p,$ is said to be $T$-stable or stable with respect to $T$, if for $\epsilon
_{n}=\left \Vert t_{n+1}-f(T,t_{n})\right \Vert ,$ $n=0,1,2,3......,$ we
have 
\begin{equation*}
\lim_{n\rightarrow \infty }\epsilon _{n}=0\text{ }\Longleftrightarrow \text{ 
}\lim_{n\rightarrow \infty }t_{n}=p.
\end{equation*}
\end{Definition}

\begin{Definition}
\cite{4aa} Let $T,\overset{\sim }{T}:X\rightarrow X$ be two operators. We
say that $\overset{\sim }{T}$ is an approximate operator for $T$ if, for
some $\varepsilon >0$, we have 
\begin{equation*}
\left \Vert Tx-\overset{\sim }{T}x\right \Vert \leq \varepsilon ,
\end{equation*}%
for all $x\in X$.
\end{Definition}

\begin{Lemma}
\label{lemma iff}\cite{19b} Let $\{ \psi _{n}\}_{n=0}^{\infty }$ and $\{
\varphi _{n}\}_{n=0}^{\infty }$ be nonnegative real sequences satisfying the
following inequality:%
\begin{equation*}
\psi _{n+1}\leq (1-\phi _{n})\psi _{n}+\varphi _{n},
\end{equation*}%
where $\phi _{n}\in (0,1),$ for all $n\in 
\mathbb{N}
,$ $\sum \limits_{n=0}^{\infty }\phi _{n}=\infty $ and $\frac{\varphi _{n}}{%
\phi _{n}}\rightarrow 0$ as $n\rightarrow \infty ,$ then $\lim_{n\rightarrow
\infty }\psi _{n}=0.$
\end{Lemma}

\begin{Lemma}
\label{lem iff}\cite{18a} Let $\{ \psi _{n}\}_{n=0}^{\infty }$ be
nonnegative real sequences for which one assumes there exists $n_{0}\in 
\mathbb{N}
$, such that for all $n\geq n_{0},$ the following inequality satisfies:%
\begin{equation*}
\psi _{n+1}\leq (1-\phi _{n})\psi _{n}+\phi _{n}\varphi _{n},
\end{equation*}%
where $\phi _{n}\in (0,1),$ for all $n\in 
\mathbb{N}
,$ $\sum \limits_{n=0}^{\infty }\phi _{n}=\infty $ and $\varphi _{n}\geq 0,$
for all $n\in 
\mathbb{N}
,$ then 
\begin{equation*}
0\leq \underset{n\rightarrow \infty }{\lim \sup }\psi _{n}\leq \underset{%
n\rightarrow \infty }{\lim \sup \varphi _{n}.}
\end{equation*}
\end{Lemma}

\section{$K$ iteration Process and its Convergence Analysis}

Through out this section we have $n\geq 0$ and $\{ \alpha _{n}\}$ and $\{
\beta _{n}\}$ are real sequences in $[0,1].$

Gursoy and Karakaya in \cite{7a} introduced new iteration process called
"Picard-S iteration process", as follow

\begin{equation}
\left \{ 
\begin{array}{c}
u_{0}\in C \\ 
w_{n}=(1-\beta _{n})u_{n}+\beta _{n}Tu_{n} \\ 
v_{n}=(1-\alpha _{n})Tu_{n}+\alpha _{n}Tw_{n} \\ 
u_{n+1}=Tv_{n}%
\end{array}%
\right.  
\end{equation}%
They proved that the Picard-S iteration process can be used to approximate
the fixed point of contraction mappings. Also, by providing an example, it
is shown that the Picard-S iteration process converge faster than all
Picard, Mann, Ishikawa, Noor, SP, CR, S, S$^{\ast },$ Abbas, Normal-S and
Two-step Mann iteration process.

After this Karakaya et. al. in \cite{8a} introduced a new two step iteration
process, with the claim that it is even faster than Picard-S iteration
process, as follow%
\begin{equation}
\left \{ 
\begin{array}{c}
u_{0}\in C \\ 
v_{n}=T((1-\beta _{n})u_{n}+\beta _{n}Tu_{n}) \\ 
u_{n+1}=T((1-\alpha _{n})v_{n}+\alpha _{n}Tv_{n})%
\end{array}%
\right.  
\end{equation}

Recently Thakur et. al. \cite{19a} used the following new iteration process,
we will call it "Thakur New iteration process", 
\begin{equation}
\left \{ 
\begin{array}{c}
u_{0}\in C \\ 
w_{n}=(1-\beta _{n})u_{n}+\beta _{n}Tu_{n} \\ 
v_{n}=T((1-\alpha _{n})u_{n}+\alpha _{n}w_{n}) \\ 
u_{n+1}=Tv_{n}%
\end{array}%
\right.  
\end{equation}%
With the help of numerical example they proved that their new iteration
process i.e."Thakur New iteration process" is faster than Picard, Mann,
Ishikawa, Agarwal, Noor and Abbas iteration process for some class of
mappings.

\begin{Problem}
Is it possible to develop an iteration process whose rate of convergence is
even faster than the iteration processes $(1)$, $(2)$ and $(3)$?
\end{Problem}

To answer this, we introduce the following new iteration process known as "$%
K $ Iteration Process"%
\begin{equation}
\left \{ 
\begin{array}{c}
x_{0}\in C \\ 
z_{n}=(1-\beta _{n})x_{n}+\beta _{n}Tx_{n} \\ 
y_{n}=T((1-\alpha _{n})Tx_{n}+\alpha _{n}Tz_{n}) \\ 
x_{n+1}=Ty_{n}%
\end{array}%
\right.  
\end{equation}

We will prove that our new iteration process $(4)$ is stable and have a
good speed of convergence comparatively to other iteration processes. Also
the data dependence result for fixed point of contraction mappings with the
help of the new iteration process is proved.

\begin{Theorem}
\label{th3.1}Let $C$ be a nonempty closed convex subset of a Banach space $X$
and $T:C\rightarrow C$ be a contraction mapping. Let $\{x_{n}\}_{n=0}^{%
\infty }$ be an iterative sequence generated by $(4)$ with real sequences $%
\{ \alpha _{n}\}_{n=0}^{\infty }$ \ and $\{ \beta _{n}\}_{n=0}^{\infty }$ in 
$[0,1]$ satisfying $\sum \limits_{n=0}^{\infty }\alpha _{n}\beta _{n}=\infty 
$. Then $\{x_{n}\}_{n=0}^{\infty }$ converge strongly to a unique fixed
point of $T.$
\end{Theorem}

\begin{proof}
The well-known Banach theorem guarantees the existence and uniqueness of
fixed point $p$. We will show that $x_{n}\rightarrow p$ for $n\rightarrow
\infty $. From $(4)$ we have 
\begin{eqnarray}
\left \Vert z_{n}-p\right \Vert &=&\left \Vert (1-\beta _{n})x_{n}+\beta
_{n}Tx_{n}-(1-\beta _{n}+\beta _{n})p\right \Vert  \notag \\
&\leq &(1-\beta _{n})\left \Vert x_{n}-p\right \Vert +\beta _{n}\left \Vert
Tx_{n}-Tp\right \Vert  \notag \\
&\leq &(1-\beta _{n})\left \Vert x_{n}-p\right \Vert +\beta _{n}\theta \left
\Vert x_{n}-p\right \Vert  \notag \\
&=&(1-\beta _{n}(1-\theta ))\left \Vert x_{n}-p\right \Vert .  
\end{eqnarray}

Similarly,
\begin{eqnarray}
\left \Vert y_{n}-p\right \Vert &=&\left \Vert T((1-\alpha
_{n})Tx_{n}+\alpha _{n}Tz_{n})-Tp\right \Vert  \notag \\
&\leq &\theta \left \Vert (1-\alpha _{n})Tx_{n}+\alpha _{n}Tz_{n}-p\right
\Vert  \notag \\
&\leq &\theta \lbrack (1-\alpha _{n})\left \Vert Tx_{n}-p\right \Vert
+\alpha _{n}\left \Vert Tz_{n}-p\right \Vert ]  \notag \\
&\leq &\theta \lbrack (1-\alpha _{n})\theta \left \Vert x_{n}-p\right \Vert
+\alpha _{n}\theta \left \Vert z_{n}-p\right \Vert ]  \notag \\
&\leq &\theta ^{2}[(1-\alpha _{n})\left \Vert x_{n}-p\right \Vert +\alpha
_{n}\left \Vert z_{n}-p\right \Vert ]  \notag \\
&\leq &\theta ^{2}((1-\alpha _{n})\left \Vert x_{n}-p\right \Vert +\alpha
_{n}(1-\beta _{n}(1-\theta ))\left \Vert x_{n}-p\right \Vert )  \notag \\
&=&\theta ^{2}(1-\alpha _{n}\beta _{n}(1-\theta ))\left \Vert x_{n}-p\right
\Vert .  
\end{eqnarray}

Hence%
\begin{eqnarray}
\left \Vert x_{n+1}-p\right \Vert &=&\left \Vert Ty_{n}-p\right \Vert  \notag
\\
&\leq &\theta \left \Vert y_{n}-p\right \Vert  \notag \\
&\leq &\theta ^{3}(1-\alpha _{n}\beta _{n}(1-\theta ))\left \Vert
x_{n}-p\right \Vert .  
\end{eqnarray}

Repetition of above processes gives the following inequalities%
\begin{equation}
\left \{ 
\begin{array}{c}
\left \Vert x_{n+1}-p\right \Vert \leq \theta ^{3}(1-\alpha _{n}\beta
_{n}(1-\theta ))\left \Vert x_{n}-p\right \Vert \\ 
\left \Vert x_{n}-p\right \Vert \leq \theta ^{3}(1-\alpha _{n-1}\beta
_{n-1}(1-\theta ))\left \Vert x_{n-1}-p\right \Vert \\ 
\left \Vert x_{n-1}-p\right \Vert \leq \theta ^{3}(1-\alpha _{n-2}\beta
_{n-2}(1-\theta ))\left \Vert x_{n-2}-p\right \Vert \\ 
: \\ 
: \\ 
\left \Vert x_{1}-p\right \Vert \leq \theta ^{3}(1-\alpha _{0}\beta
_{0}(1-\theta ))\left \Vert x_{0}-p\right \Vert .%
\end{array}%
\right.  
\end{equation}

From $(8)$ we can easily derive 
\begin{equation}
\left \Vert x_{n+1}-p\right \Vert \leq \left \Vert x_{0}-p\right \Vert
\theta ^{3(n+1)}\prod \limits_{k=0}^{n}(1-\alpha _{k}\beta _{k}(1-\theta )),
\end{equation}%
where $1-\alpha _{k}\beta _{k}(1-\theta )<1,$ because $\theta \in (0,1)$ and 
$\alpha _{n},\beta _{n}$ $\in $ $[0,1]$, for all $n\in 
\mathbb{N}
.$

Since we know that $1-x\leq e^{-x}$ for all $x\in \lbrack 0,1]$ $,$ so from $%
(9)$ we get%
\begin{equation}
\left \Vert x_{n+1}-p\right \Vert \leq \left \Vert x_{0}-p\right \Vert
\theta ^{3(n+1)}e^{-(1-\theta )\sum_{k=0}^{n}\alpha _{k}\beta _{k}}. 
\end{equation}

Taking the limit of both sides of $(10),$ we get $\lim_{n\rightarrow
\infty }\left \Vert x_{n}-p\right \Vert =0$, i.e. $x_{n}\rightarrow p$ for $%
n\rightarrow \infty $, as required.
\end{proof}

\begin{Theorem}
Let $C$ be a nonempty closed convex subset of a Banach space $X$ and $%
T:C\rightarrow C$ be a contraction mapping. Let $\{x_{n}\}_{n=0}^{\infty }$
be an iterative sequence generated by $(4)$ with real sequences $\{ \alpha
_{n}\}_{n=0}^{\infty }$ \ and $\{ \beta _{n}\}_{n=0}^{\infty }$ in $[0,1]$
satisfying $\sum \limits_{n=0}^{\infty }\alpha _{n}\beta _{n}=\infty $. Then
the iterative process $(4)$ is $T$-stable.
\end{Theorem}

\begin{proof}
Let $\{t_{n}\}_{n=0}^{\infty }$ $\subset X$ be any arbitrary sequence in $C.$
Let the sequence generated by $(4)$ is $x_{n+1}=f(T,x_{n})$ converging to
unique fixed point $p$ (by Theorem \ref{th3.1}) and $\epsilon
_{n}=\left
\Vert t_{n+1}-f(T,t_{n})\right \Vert .$ We will prove that $%
\lim_{n\rightarrow \infty }\epsilon _{n}=0$ $\Longleftrightarrow $ $%
\lim_{n\rightarrow \infty }t_{n}=p.$

Let $\lim_{n\rightarrow \infty }\epsilon _{n}=0,$ we have%
\begin{eqnarray*}
\left \Vert t_{n+1}-p\right \Vert &\leq &\left \Vert
t_{n+1}-f(T,t_{n})\right \Vert +\left \Vert f(T,t_{n})-p\right \Vert \\
&=&\epsilon _{n}+\left \Vert T(T((1-\beta _{n})Tt_{n}+\beta _{n}T((1-\alpha
_{n})t_{n}+\alpha _{n}Tt_{n})))-p\right \Vert \\
&\leq &\theta ^{3}(1-\alpha _{n}\beta _{n}(1-\theta ))\left \Vert
t_{n}-p\right \Vert +\epsilon _{n}.
\end{eqnarray*}

Since $\theta \in (0,1)$, $\alpha _{n},\beta _{n}$ $\in $ $[0,1]$, for all $%
n\in 
\mathbb{N}
$ and $\lim_{n\rightarrow \infty }\epsilon _{n}=0,$ so the above inequality
together with Lemma \ref{lemma iff} leads to $\lim_{n\rightarrow \infty
}\left \Vert t_{n}-p\right \Vert =0.$ Hence $\lim_{n\rightarrow \infty
}t_{n}=p.$

Conversely let $\lim_{n\rightarrow \infty }t_{n}=p,$ we have%
\begin{eqnarray*}
\epsilon _{n} &=&\left \Vert t_{n+1}-f(T,t_{n})\right \Vert \\
&\leq &\left \Vert t_{n+1}-p\right \Vert +\left \Vert f(T,t_{n})-p\right
\Vert \\
&\leq &\left \Vert t_{n+1}-p\right \Vert +\theta ^{3}(1-\alpha _{n}\beta
_{n}(1-\theta ))\left \Vert t_{n}-p\right \Vert .
\end{eqnarray*}%
This implies that $\lim_{n\rightarrow \infty }\epsilon _{n}=0.$

Hence $(4)$ is stable with respect to $T.$
\end{proof}

\begin{Theorem}
\label{th3.2}Let $C$ be a nonempty closed convex subset of a Banach space $X$
and $T:C\rightarrow C$ a contraction mapping with fixed point $p$. For given 
$u_{0}=x_{0}\in C,$ let $\{u_{n}\}_{n=0}^{\infty }$ and $\{x_{n}\}_{n=0}^{%
\infty }$ be an iterative sequences generated by $(1)$ and $(4)$
respectively, with real sequences $\{ \alpha _{n}\}_{n=0}^{\infty }$ \ and $%
\{ \beta _{n}\}_{n=0}^{\infty }$ in $[0,1]$ satisfying

$(i).$ $\alpha \leq \alpha _{n}<1$ and $\beta \leq \beta _{n}<1,$ for some $%
\alpha ,\beta >0$ and for all $n\in 
\mathbb{N}
.$

Then $\{x_{n}\}_{n=0}^{\infty }$ converge to $p$ faster than $%
\{u_{n}\}_{n=0}^{\infty }$ does.
\end{Theorem}

\begin{proof}
From $(9)$ of Theorem \ref{th3.1}$,$ we have 
\begin{equation}
\left \Vert x_{n+1}-p\right \Vert \leq \left \Vert x_{0}-p\right \Vert
\theta ^{3(n+1)}\prod \limits_{k=0}^{n}(1-\alpha _{k}\beta _{k}(1-\theta )).
\end{equation}

The following inequality is due to \cite[2.5]{7a} which is obtained from $%
(1),$ also converging to unique fixed point $p$ \cite[Theorem 1]{7a},
\begin{equation}
\left \Vert u_{n+1}-p\right \Vert \leq \left \Vert u_{0}-p\right \Vert
\theta ^{2(n+1)}\prod \limits_{k=0}^{n}(1-\alpha _{k}\beta _{k}(1-\theta )).
\end{equation}%
.

Together with assumption $(i)$, $(11)$ implies that,%
\begin{eqnarray}
\left \Vert x_{n+1}-p\right \Vert &\leq &\left \Vert x_{0}-p\right \Vert
\theta ^{3(n+1)}\prod \limits_{k=0}^{n}(1-\alpha \beta (1-\theta ))  \notag
\\
&=&\left \Vert x_{0}-p\right \Vert \theta ^{3(n+1)}(1-\alpha \beta (1-\theta
))^{n+1}.  
\end{eqnarray}

Similarly $(12)$ together with assumption $(i)$ leads to,%
\begin{eqnarray}
\left \Vert u_{n+1}-p\right \Vert &\leq &\left \Vert u_{0}-p\right \Vert
\theta ^{2(n+1)}\prod \limits_{k=0}^{n}(1-\alpha \beta (1-\theta ))  \notag
\\
&=&\left \Vert u_{0}-p\right \Vert \theta ^{2(n+1)}(1-\alpha \beta (1-\theta
))^{n+1}.  
\end{eqnarray}

Define%
\begin{equation*}
a_{n}=\left \Vert x_{0}-p\right \Vert \theta ^{3(n+1)}(1-\alpha \beta
(1-\theta ))^{n+1},
\end{equation*}%
and 
\begin{equation*}
b_{n}=\left \Vert u_{0}-p\right \Vert \theta ^{2(n+1)}(1-\alpha \beta
(1-\theta ))^{n+1},
\end{equation*}%
then%
\begin{eqnarray}
\Psi _{n} &=&\frac{a_{n}}{b_{n}}  \notag \\
&=&\frac{\left \Vert x_{0}-p\right \Vert \theta ^{3(n+1)}(1-\alpha \beta
(1-\theta ))^{n+1}}{\left \Vert u_{0}-p\right \Vert \theta
^{2(n+1)}(1-\alpha \beta (1-\theta ))^{n+1}}  \notag \\
&=&\theta ^{n+1}.  
\end{eqnarray}

Since $\lim_{n\rightarrow \infty }\frac{\Psi _{n+1}}{\Psi _{n}}%
=\lim_{n\rightarrow \infty }\frac{\theta ^{n+2}}{\theta ^{n+1}}=\theta <1,$
so by ratio test $\sum \limits_{n=0}^{\infty }\Psi _{n}<\infty .$ Hence from 
$(15)$ we have,%
\begin{equation*}
\lim_{n\rightarrow \infty }\frac{a_{n}}{b_{n}}=\lim_{n\rightarrow \infty
}\Psi _{n}=0,
\end{equation*}%
which implies that $\{x_{n}\}_{n=0}^{\infty }$ is faster than $%
\{u_{n}\}_{n=0}^{\infty }.$
\end{proof}

We are now able to establish the following data dependence result.

\begin{Theorem}
Let $\overset{\sim }{T}$ be an approximate operator of a contraction mapping 
$T$. Let $\{x_{n}\}_{n=0}^{\infty }$ be an iterative sequence generated by $%
(4)$ for $T$ and define an iterative sequence $\{ \overset{\sim }{x}%
_{n}\}_{n=0}^{\infty }$ as follows%
\begin{equation}
\left \{ 
\begin{array}{c}
\overset{\sim }{x}_{0}\in C \\ 
\overset{\sim }{z}_{n}=(1-\beta _{n})\overset{\sim }{x}_{n}+\beta _{n}%
\overset{\sim }{T}\overset{\sim }{x}_{n} \\ 
\overset{\sim }{y}_{n}=\overset{\sim }{T}((1-\alpha _{n})\overset{\sim }{T}%
\overset{\sim }{x}_{n}+\alpha _{n}\overset{\sim }{T}\overset{\sim }{z}_{n})
\\ 
\overset{\sim }{x}_{n+1}=\overset{\sim }{T}\overset{\sim }{y}_{n},%
\end{array}%
\right.  
\end{equation}%
with real sequences $\{ \alpha _{n}\}_{n=0}^{\infty }$ \ and $\{ \beta
_{n}\}_{n=0}^{\infty }$ in $[0,1]$ satisfying $(i).$ $\frac{1}{2}\leq \alpha
_{n}\beta _{n},$ for all $n\in 
\mathbb{N}
,$ and $(ii).$ $\sum \limits_{n=0}^{\infty }\alpha _{n}\beta _{n}=\infty .$
If $Tp=p$ and $\overset{\sim }{T}\overset{\sim }{p}=\overset{\sim }{p}$ such
that $\lim_{n\rightarrow \infty }\overset{\sim }{x}_{n}=\overset{\sim }{p},$
then we have%
\begin{equation*}
\left \Vert p-\overset{\sim }{p}\right \Vert \leq \frac{7\varepsilon }{%
1-\theta },
\end{equation*}%
where $\varepsilon >0$ is a fixed number.
\end{Theorem}

\begin{proof}
It follows from $(4)$ and $(16)$ that%
\begin{eqnarray}
\left \Vert z_{n}-\overset{\sim }{z}_{n}\right \Vert &=&\left \Vert (1-\beta
_{n})x_{n}+\beta _{n}Tx_{n}-(1-\beta _{n})\overset{\sim }{x}_{n}-\beta _{n}%
\overset{\sim }{T}\overset{\sim }{x}_{n}\right \Vert  \notag \\
&\leq &(1-\beta _{n})\left \Vert x_{n}-\overset{\sim }{x}_{n}\right \Vert
+\beta _{n}\left \Vert Tx_{n}-\overset{\sim }{T}\overset{\sim }{x}_{n}\right
\Vert  \notag \\
&\leq &(1-\beta _{n})\left \Vert x_{n}-\overset{\sim }{x}_{n}\right \Vert
+\beta _{n}\left \{ \left \Vert Tx_{n}-T\overset{\sim }{x}_{n}\right \Vert
+\left \Vert T\overset{\sim }{x}_{n}-\overset{\sim }{T}\overset{\sim }{x}%
_{n}\right \Vert \right \}  \notag \\
&\leq &(1-\beta _{n}(1-\theta ))\left \Vert x_{n}-\overset{\sim }{x}%
_{n}\right \Vert +\beta _{n}\varepsilon .  
\end{eqnarray}

Using $(17)$, we have%
\begin{eqnarray}
\left \Vert y_{n}-\overset{\sim }{y}_{n}\right \Vert &=&\left \Vert
T((1-\alpha _{n})Tx_{n}+\alpha _{n}Tz_{n})-\overset{\sim }{T}((1-\alpha _{n})%
\overset{\sim }{T}\overset{\sim }{x}_{n}+\alpha _{n}\overset{\sim }{T}%
\overset{\sim }{z}_{n})\right \Vert  \notag \\
&\leq &\left \Vert T((1-\alpha _{n})Tx_{n}+\alpha _{n}Tz_{n})-T((1-\alpha
_{n})\overset{\sim }{T}\overset{\sim }{x}_{n}+\alpha _{n}\overset{\sim }{T}%
\overset{\sim }{z}_{n})\right \Vert  \notag \\
&&+\left \Vert T((1-\alpha _{n})\overset{\sim }{T}\overset{\sim }{x}%
_{n}+\alpha _{n}\overset{\sim }{T}\overset{\sim }{z}_{n})-\overset{\sim }{T}%
((1-\alpha _{n})\overset{\sim }{T}\overset{\sim }{x}_{n}+\alpha _{n}\overset{%
\sim }{T}\overset{\sim }{z}_{n})\right \Vert  \notag \\
&\leq &\theta \left \Vert (1-\alpha _{n})Tx_{n}+\alpha _{n}Tz_{n}-(1-\alpha
_{n})\overset{\sim }{T}\overset{\sim }{x}_{n}-\alpha _{n}\overset{\sim }{T}%
\overset{\sim }{z}_{n}\right \Vert +\varepsilon  \notag \\
&\leq &\theta \left[ (1-\alpha _{n})\left \Vert Tx_{n}-\overset{\sim }{T}%
\overset{\sim }{x}_{n}\right \Vert +\alpha _{n}\left \Vert Tz_{n}-\overset{%
\sim }{T}\overset{\sim }{z}_{n}\right \Vert \right] +\varepsilon  \notag \\
&\leq &\theta \left[ 
\begin{array}{c}
(1-\alpha _{n})\left \{ \left \Vert Tx_{n}-T\overset{\sim }{x}_{n}\right
\Vert +\left \Vert T\overset{\sim }{x}_{n}-\overset{\sim }{T}\overset{\sim }{%
x}_{n}\right \Vert \right \} \\ 
+\alpha _{n}\left \{ \left \Vert Tz_{n}-T\overset{\sim }{z}_{n}\right \Vert
+\left \Vert T\overset{\sim }{z}_{n}-\overset{\sim }{T}\overset{\sim }{z}%
_{n}\right \Vert \right \}%
\end{array}%
\right] +\varepsilon  \notag \\
&\leq &\theta \left[ (1-\alpha _{n})\left \{ \theta \left \Vert x_{n}-%
\overset{\sim }{x}_{n}\right \Vert +\varepsilon \right \} +\alpha _{n}\left
\{ \theta \left \Vert z_{n}-\overset{\sim }{z}_{n}\right \Vert +\varepsilon
\right \} \right] +\varepsilon  \notag \\
&=&\theta \left[ (1-\alpha _{n})\theta \left \Vert x_{n}-\overset{\sim }{x}%
_{n}\right \Vert +\alpha _{n}\theta \left \Vert z_{n}-\overset{\sim }{z}%
_{n}\right \Vert +\varepsilon \right] +\varepsilon  \notag \\
&\leq &\theta \left[ 
\begin{array}{c}
(1-\alpha _{n})\theta \left \Vert x_{n}-\overset{\sim }{x}_{n}\right \Vert
\\ 
+\alpha _{n}\theta \left \{ (1-\beta _{n}(1-\theta ))\left \Vert x_{n}-%
\overset{\sim }{x}_{n}\right \Vert +\beta _{n}\varepsilon \right \}
+\varepsilon%
\end{array}%
\right] +\varepsilon  \notag \\
&=&\theta ^{2}(1-\alpha _{n}\beta _{n}(1-\theta )\left \Vert x_{n}-\overset{%
\sim }{x}_{n}\right \Vert +\theta \varepsilon (1+\theta \alpha _{n}\beta
_{n})+\varepsilon .  
\end{eqnarray}

Similarly using $(18)$, we have%
\begin{eqnarray}
\left \Vert x_{n+1}-\overset{\sim }{x}_{n+1}\right \Vert &=&\left \Vert
Ty_{n}-\overset{\sim }{T}\overset{\sim }{y}_{n}\right \Vert  \notag \\
&\leq &\theta \left \Vert y_{n}-\overset{\sim }{y}_{n}\right \Vert
+\varepsilon  \notag \\
&\leq &\theta ^{3}(1-\alpha _{n}\beta _{n}(1-\theta )\left \Vert x_{n}-%
\overset{\sim }{x}_{n}\right \Vert +\theta ^{2}\varepsilon (1+\theta \alpha
_{n}\beta _{n})+\theta \varepsilon +\varepsilon  \notag \\
&\leq &(1-\alpha _{n}\beta _{n}(1-\theta )\left \Vert x_{n}-\overset{\sim }{x%
}_{n}\right \Vert +\varepsilon (1+\theta \alpha _{n}\beta _{n})+\varepsilon
+\varepsilon  \notag \\
&\leq &(1-\alpha _{n}\beta _{n}(1-\theta )\left \Vert x_{n}-\overset{\sim }{x%
}_{n}\right \Vert +\alpha _{n}\beta _{n}\varepsilon +3\varepsilon  \notag \\
&=&(1-\alpha _{n}\beta _{n}(1-\theta )\left \Vert x_{n}-\overset{\sim }{x}%
_{n}\right \Vert +\alpha _{n}\beta _{n}\varepsilon  \notag \\
&&+3(1-\alpha _{n}\beta _{n}+\alpha _{n}\beta _{n})\varepsilon . 
\end{eqnarray}

By assumption $(i)$ we have $1-\alpha _{n}\beta _{n}\leq \alpha _{n}\beta
_{n}.$ Using this together with $(19),$ we get%
\begin{eqnarray}
\left \Vert x_{n+1}-\overset{\sim }{x}_{n+1}\right \Vert &\leq &(1-\alpha
_{n}\beta _{n}(1-\theta )\left \Vert x_{n}-\overset{\sim }{x}_{n}\right
\Vert +7\alpha _{n}\beta _{n}\varepsilon  \notag \\
&=&(1-\alpha _{n}\beta _{n}(1-\theta )\left \Vert x_{n}-\overset{\sim }{x}%
_{n}\right \Vert  \notag \\
&&+\alpha _{n}\beta _{n}(1-\theta )\frac{7\varepsilon }{1-\theta }. 
\end{eqnarray}

Let $\psi _{n}=\left \Vert x_{n}-\overset{\sim }{x}_{n}\right \Vert ,$ $\phi
_{n}=\alpha _{n}\beta _{n}(1-\theta ),$ $\varphi _{n}=\frac{7\varepsilon }{%
1-\theta },$ then from Lemma \ref{lem iff} together with $(20)$, we get 
\begin{equation}
0\leq \underset{n\rightarrow \infty }{\lim \sup }\left \Vert x_{n}-\overset{%
\sim }{x}_{n}\right \Vert \leq \underset{n\rightarrow \infty }{\lim \sup }%
\frac{7\varepsilon }{1-\theta }.  
\end{equation}

Since by Theorem \ref{th3.1} we have $\lim_{n\rightarrow \infty }x_{n}=p$
and by assumption we have $\lim_{n\rightarrow \infty }\overset{\sim }{x}_{n}=%
\overset{\sim }{p}.$ Using these together with $(21),$ we get%
\begin{equation*}
\left \Vert p-\overset{\sim }{p}\right \Vert \leq \frac{7\varepsilon }{%
1-\theta },
\end{equation*}%
as required.
\end{proof}

\section{Convergence results for Suzuki generalized nonexpansive mappings}

In this section, we prove weak and strong convergence theorems of a sequence
generated by $K$ iteration process for Suzuki generalized nonexpansive
mappings in the setting of uniformly convex Banach spaces.

\begin{Lemma}
\label{lem1}Let $C$ be a nonempty closed convex subset of a Banach space $X$%
, and let $T:C\rightarrow C$ be a mapping satisfying condition $(C)$ with $%
F(T)\neq \emptyset $. For arbitrary chosen $x_{0}\in C,$ let the sequence $%
\{x_{n}\}$ be generated by $(4)$, then $\lim_{n\rightarrow \infty
}\left
\Vert x_{n}-p\right \Vert $ exists for any $p\in F(T).$
\end{Lemma}

\begin{proof}
Let $p\in F(T)$ and $z\in C$. Since $T$ satisfies condition $(C)$, so 
\begin{equation*}
\frac{1}{2}\left \Vert p-Tp\right \Vert =0\leq \left \Vert p-z\right \Vert 
\text{ implies that }\left \Vert Tp-Tz\right \Vert \leq \left \Vert
p-z\right \Vert .
\end{equation*}

So by Proposition \ref{p1}$(ii)$, we have,%
\begin{eqnarray}
\left \Vert z_{n}-p\right \Vert &=&\left \Vert (1-\beta _{n})x_{n}+\beta
_{n}Tx_{n}-p\right \Vert  \notag \\
&\leq &(1-\beta _{n})\left \Vert x_{n}-p\right \Vert +\beta _{n}\left \Vert
Tx_{n}-p\right \Vert  \notag \\
&\leq &(1-\beta _{n})\left \Vert x_{n}-p\right \Vert +\beta _{n}\left \Vert
x_{n}-p\right \Vert  \notag \\
&=&\left \Vert x_{n}-p\right \Vert .  
\end{eqnarray}

So by using $(22)$ we get,%
\begin{eqnarray}
\left \Vert y_{n}-p\right \Vert &=&\left \Vert T((1-\alpha
_{n})Tx_{n}+\alpha _{n}Tz_{n})-p\right \Vert  \notag \\
&\leq &\left \Vert (1-\alpha _{n})Tx_{n}+\alpha _{n}Tz_{n}-p\right \Vert 
\notag \\
&\leq &(1-\alpha _{n})\left \Vert Tx_{n}-p\right \Vert +\alpha _{n}\left
\Vert Tz_{n}-p\right \Vert  \notag \\
&\leq &(1-\alpha _{n})\left \Vert x_{n}-p\right \Vert +\alpha _{n}\left
\Vert z_{n}-p\right \Vert  \notag \\
&\leq &(1-\alpha _{n})\left \Vert x_{n}-p\right \Vert +\alpha _{n}\left
\Vert x_{n}-p\right \Vert  \notag \\
&=&\left \Vert x_{n}-p\right \Vert .  
\end{eqnarray}

Similarly, by using $(23)$ we have,%
\begin{eqnarray}
\left \Vert x_{n+1}-p\right \Vert &=&\left \Vert Ty_{n}-p\right \Vert  \notag
\\
&\leq &\left \Vert y_{n}-p\right \Vert  \notag \\
&\leq &\left \Vert x_{n}-p\right \Vert .  
\end{eqnarray}

This implies that $\{ \left \Vert x_{n}-p\right \Vert \}$ is bounded and
non-increasing for all $p\in F(T).$ Hence $\lim_{n\rightarrow \infty
}\left
\Vert x_{n}-p\right \Vert $ exists, as required.
\end{proof}

\begin{Theorem}
\label{th 1}Let $C$ be a nonempty closed convex subset of a uniformly convex
Banach space $X$, and let $T:C\rightarrow C$ be a mapping satisfying
condition $(C)$. For arbitrary chosen $x_{0}\in C$, let the sequence $%
\{x_{n}\}$ be generated by $(4)$ for all $n\geq 1$, where $\{ \alpha
_{n}\} $ and $\{ \beta _{n}\}$ are sequence of real numbers in $[a,b]$ for
some $a,b $ with $0<a\leq b<1$. Then $F(T)\neq \emptyset $ if and only if $%
\{x_{n}\}$ is bounded and $\lim_{n\rightarrow \infty }\left \Vert
Tx_{n}-x_{n}\right
\Vert =0 $.
\end{Theorem}

\begin{proof}
Suppose $F(T)\neq \emptyset $ and let $p\in F(T)$. Then, by Lemma \ref{lem1}%
, $\lim_{n\rightarrow \infty }\left \Vert x_{n}-p\right \Vert $ exists and $%
\{x_{n}\}$ is bounded. Put%
\begin{equation}
\lim_{n\rightarrow \infty }\left \Vert x_{n}-p\right \Vert =r.  
\end{equation}

From $(22)$ and $(25)$, we have%
\begin{equation}
\underset{n\rightarrow \infty }{\lim \sup }\left \Vert z_{n}-p\right \Vert
\leq \underset{n\rightarrow \infty }{\lim \sup }\left \Vert x_{n}-p\right
\Vert =r.  
\end{equation}

By Proposition \ref{p1}$(ii)$ we have%
\begin{equation}
\underset{n\rightarrow \infty }{\lim \sup }\left \Vert Tx_{n}-p\right \Vert
\leq \underset{n\rightarrow \infty }{\lim \sup }\left \Vert x_{n}-p\right
\Vert =r.  
\end{equation}

On the other hand 
\begin{eqnarray*}
\left \Vert x_{n+1}-p\right \Vert &=&\left \Vert Ty_{n}-p\right \Vert \\
&\leq &\left \Vert y_{n}-p\right \Vert \\
&=&\left \Vert T((1-\alpha _{n})Tx_{n}+\alpha _{n}Tz_{n})-p\right \Vert \\
&\leq &\left \Vert (1-\alpha _{n})Tx_{n}+\alpha _{n}Tz_{n}-p\right \Vert \\
&\leq &(1-\alpha _{n})\left \Vert Tx_{n}-p\right \Vert +\alpha _{n}\left
\Vert Tz_{n}-p\right \Vert \\
&\leq &(1-\alpha _{n})\left \Vert x_{n}-p\right \Vert +\alpha _{n}\left
\Vert z_{n}-p\right \Vert \\
&\leq &\left \Vert x_{n}-p\right \Vert -\alpha _{n}\left \Vert x_{n}-p\right
\Vert +\alpha _{n}\left \Vert z_{n}-p\right \Vert .
\end{eqnarray*}

This implies that%
\begin{equation*}
\frac{\left \Vert x_{n+1}-p\right \Vert -\left \Vert x_{n}-p\right \Vert }{%
\alpha _{n}}\leq \left \Vert z_{n}-p\right \Vert -\left \Vert x_{n}-p\right
\Vert .
\end{equation*}

So%
\begin{equation*}
\left \Vert x_{n+1}-p\right \Vert -\left \Vert x_{n}-p\right \Vert \leq 
\frac{\left \Vert x_{n+1}-p\right \Vert -\left \Vert x_{n}-p\right \Vert }{%
\alpha _{n}}\leq \left \Vert z_{n}-p\right \Vert -\left \Vert x_{n}-p\right
\Vert ,
\end{equation*}

implies that 
\begin{equation*}
\left \Vert x_{n+1}-p\right \Vert \leq \left \Vert z_{n}-p\right \Vert .
\end{equation*}

Therefore 
\begin{equation}
r\leq \underset{n\rightarrow \infty }{\lim \inf }\left \Vert z_{n}-p\right
\Vert .  
\end{equation}

From $(26)$ and $(28)$ we get,%
\begin{eqnarray}
r &=&\underset{n\rightarrow \infty }{\lim }\left \Vert z_{n}-p\right \Vert 
\notag \\
&=&\underset{n\rightarrow \infty }{\lim }\left \Vert (1-\beta
_{n})x_{n}+\beta _{n}Tx_{n}-p\right \Vert  \notag \\
&=&\underset{n\rightarrow \infty }{\lim }\left \Vert \beta
_{n}(Tx_{n}-p)+(1-\beta _{n})(x_{n}-p)\right \Vert .  
\end{eqnarray}

From $(25)$, $(27)$, $(29)$ togather with Lemma \ref{l3}, we have, $\underset{%
n\rightarrow \infty }{\lim }\left \Vert Tx_{n}-x_{n}\right \Vert =0.$

Conversely, suppose that $\{x_{n}\}$ is bounded and $\lim_{n\rightarrow
\infty }\left \Vert Tx_{n}-x_{n}\right \Vert =0.$ Let $p\in A(C,\{x_{n}\})$.
By Proposition $\ref{p1}(iii)$, we have

\begin{eqnarray*}
r(Tp,\{x_{n}\}) &=&\underset{n\rightarrow \infty }{\lim \sup }\left \Vert
x_{n}-Tp\right \Vert  \\
&\leq &\underset{n\rightarrow \infty }{\lim \sup }(3\left \Vert
Tx_{n}-x_{n}\right \Vert +\left \Vert x_{n}-p\right \Vert ) \\
&\leq &\underset{n\rightarrow \infty }{\lim \sup }\left \Vert
x_{n}-p\right \Vert  \\
&=&r(p,\{x_{n}\}).
\end{eqnarray*}

This implies that $Tp\in A(C,\{x_{n}\})$. Since $X$ is uniformly convex, $%
A(C,\{x_{n}\})$ is singleton, hence we have $Tp=p.$ Hence $F(T)\neq
\emptyset .$
\end{proof}

Now we are in the position to prove weak convergence theorem.

\begin{Theorem}
\label{th 2}Let $C$ be a nonempty closed convex subset of a uniformly convex
Banach space $X$ with the Opial property, and let $T:C\rightarrow C$ be a
mapping satisfying condition $(C)$. For arbitrary chosen $x_{0}\in C$, let
the sequence $\{x_{n}\}$ be generated by $(4)$ for all $n\geq 1$, where $%
\{ \alpha _{n}\}$ and $\{ \beta _{n}\}$ are sequence of real numbers in $%
[a,b]$ for some $a,b$ with $0<a\leq b<1$ such that $F(T)\neq \emptyset .$
Then $\{x_{n}\}$ converges weakly to a fixed point of $T$ .
\end{Theorem}

\begin{proof}
Since $F(T)\neq \emptyset ,$ so by Theorem \ref{th 1} we have that $%
\{x_{n}\} $ is bounded and $\lim_{n\rightarrow \infty }\left \Vert
Tx_{n}-x_{n}\right \Vert =0.$ Since $X$ is uniformly convex hence reflexive,
so by Eberlin's theorem there exists a subsequence $\{x_{n_{j}}\}$ of $%
\{x_{n}\}$ which converges weakly to some $q_{1}\in X$. Since $C$ is closed
and convex, by Mazur's theorem $q_{1}\in C$. By Lemma \ref{l1}, $q_{1}\in
F(T)$. Now, we show that $\{x_{n}\}$ converges weakly to $q_{1}$. In fact,
if this is not true, so there must exist a subsequence $\{x_{n_{k}}\}$ of $%
\{x_{n}\}$ such that $\{x_{n_{k}}\}$ converges weakly to $q_{2}\in C$ and $%
q_{2}\neq q_{1}$. By Lemma \ref{l1}, $q_{2}\in F(T)$. Since $%
\lim_{n\rightarrow \infty }\left \Vert x_{n}-p\right \Vert $ exists for all $%
p\in F(T)$. By Theorem \ref{th 1} and Opial's property, we have%
\begin{eqnarray*}
\lim_{n\rightarrow \infty }\left \Vert x_{n}-q_{1}\right \Vert
&=&\lim_{j\rightarrow \infty }\left \Vert x_{n_{j}}-q_{1}\right \Vert \\
&<&\lim_{j\rightarrow \infty }\left \Vert x_{n_{j}}-q_{2}\right \Vert \\
&=&\lim_{n\rightarrow \infty }\left \Vert x_{n}-q_{2}\right \Vert \\
&=&\lim_{k\rightarrow \infty }\left \Vert x_{n_{k}}-q_{2}\right \Vert \\
&<&\lim_{k\rightarrow \infty }\left \Vert x_{n_{k}}-q_{1}\right \Vert \\
&=&\lim_{n\rightarrow \infty }\left \Vert x_{n}-q_{1}\right \Vert ,
\end{eqnarray*}%
which is contradiction. So $q_{1}=q_{2}.$ This implies that $\{x_{n}\}$
converges weakly to a fixed point of $T$.
\end{proof}

Next we prove the strong convergence theorem.

\begin{Theorem}
Let $C$ be a nonempty compact convex subset of a uniformly convex Banach
space $X$, and let $T:C\rightarrow C$ be a mapping satisfying condition $(C)$%
. For arbitrary chosen $x_{0}\in C$, let the sequence $\{x_{n}\}$ be
generated by $(4)$ for all $n\geq 1$, where $\{ \alpha _{n}\}$ and $\{
\beta _{n}\}$ are sequence of real numbers in $[a,b]$ for some $a,b$ with $%
0<a\leq b<1.$ Then $\{x_{n}\}$ converges strongly to a fixed point of $T$ .
\end{Theorem}

\begin{proof}
By Lemma \ref{l2}, we have that $F(T)\neq \emptyset $ so by Theorem \ref{th
1} we have $\lim_{n\rightarrow \infty }\left \Vert Tx_{n}-x_{n}\right \Vert
=0. $ Since $C$ is compact, so there exists a subsequence $\{x_{n_{k}}\}$ of 
$\{x_{n}\}$ such that $\{x_{n_{k}}\}$ converges strongly to $p$ for some $%
p\in C$. By Proposition \ref{p1}$(iii)$, we have%
\begin{equation*}
\left \Vert x_{n_{k}}-Tp\right \Vert \leq 3\left \Vert
Tx_{n_{k}}-x_{n_{k}}\right \Vert +\left \Vert x_{n_{k}}-p\right \Vert ,\text{
for all }n\geq 1.
\end{equation*}

Letting $k\rightarrow \infty ,$ we get $Tp=p,$ $i.e.,$ $p\in F(T)$. Since,
by Lemma \ref{lem1}, $\lim_{n\rightarrow \infty }\left \Vert
x_{n}-p\right
\Vert $ exists for every $p\in F(T),$ so $x_{n}$ converge
strongly to $p.$
\end{proof}

Senter and Dotson \cite{18} introduced the notion of a mappings satisfying
condition $(I)$ as.

A mapping $T:C\rightarrow C$ is said to satisfy condition $(I)$, if there
exists a nondecreasing function $f:[0,\infty )\rightarrow \lbrack 0,\infty )$
with $f(0)=0$ and $f(r)>0$ for all $r>0$ such that $\left \Vert
x-Tx\right
\Vert \geq f(d(x,F(T)))$ for all $x\in C$, where $%
d(x,F(T))=\inf_{p\in F(T)}\left \Vert x-p\right \Vert $.

Now we prove the strong convergence theorem using condition $(I)$.

\begin{Theorem}
Let $C$ be a nonempty closed convex subset of a uniformly convex Banach
space $X$, and let $T:C\rightarrow C$ be a mapping satisfying condition $(C)$%
. For arbitrary chosen $x_{0}\in C$, let the sequence $\{x_{n}\}$ be
generated by $(4)$ for all $n\geq 1$, where $\{ \alpha _{n}\}$ and $\{
\beta _{n}\}$ are sequence of real numbers in $[a,b]$ for some $a,b$ with $%
0<a\leq b<1$ such that $F(T)\neq \emptyset .$ If $T$ satisfy condition $(I),$
then $\{x_{n}\}$ converges strongly to a fixed point of $T$.
\end{Theorem}

\begin{proof}
By Lemma \ref{lem1}, we have $\lim_{n\rightarrow \infty }\left \Vert
x_{n}-p\right \Vert $ exists for all $p\in F(T)$ and so $\lim_{n\rightarrow
\infty }d(x_{n},F(T))$ exists. Assume that $\lim_{n\rightarrow \infty
}\left
\Vert x_{n}-p\right \Vert =r$ for some $r\geq 0$. If $r=0$ then the
result follows. Suppose $r>0$, from the hypothesis and condition $(I)$,%
\begin{equation}
f(d(x_{n},F(T)))\leq \left \Vert Tx_{n}-x_{n}\right \Vert .  
\end{equation}

Since $F(T)\neq \emptyset ,$ so by Theorem \ref{th 2}, we have $%
\lim_{n\rightarrow \infty }\left \Vert Tx_{n}-x_{n}\right \Vert =0.$ So $%
(30)$ implies that%
\begin{equation}
\lim_{n\rightarrow \infty }f(d(x_{n},F(T)))=0.  
\end{equation}

Since $f$ is nondecreasing function, so from $(31)$ we have $%
\lim_{n\rightarrow \infty }d(x_{n},F(T))=0$. Thus, we have a subsequence $%
\{x_{n_{k}}\}$ of $\{x_{n}\}$ and a sequence $\{y_{k}\} \subset F(T)$ such
that%
\begin{equation*}
\left \Vert x_{n_{k}}-y_{k}\right \Vert <\frac{1}{2^{k}}\text{ for all }k\in 
\mathbb{N}
.
\end{equation*}

So using $(24),$we get%
\begin{equation*}
\left \Vert x_{n_{k+1}}-y_{k}\right \Vert \leq \left \Vert
x_{n_{k}}-y_{k}\right \Vert <\frac{1}{2^{k}}.
\end{equation*}

Hence%
\begin{eqnarray*}
\left \Vert y_{k+1}-y_{k}\right \Vert &\leq &\left \Vert
y_{k+1}-x_{k+1}\right \Vert +\left \Vert x_{k+1}-y_{k}\right \Vert \\
&\leq &\frac{1}{2^{k+1}}+\frac{1}{2^{k}} \\
&<&\frac{1}{2^{k-1}}\rightarrow 0,\text{ as }k\rightarrow \infty .
\end{eqnarray*}

This shows that $\{y_{k}\}$ is a Cauchy sequence in $F(T)$ and so it
converges to a point $p$. Since $F(T)$ is closed, therefore $p\in F(T)$ and
then $\{x_{n_{k}}\}$ converges strongly to $p$. Since $\lim_{n\rightarrow
\infty }\left \Vert x_{n}-p\right \Vert $ exists, we have that $%
x_{n}\rightarrow p\in F(T)$. Hence proved.
\end{proof}

\section{Numerical Example}

In order to support analytical proof of Theorem \ref{th3.2} and to
illustrate the efficiency of $K$ iteration method $(4)$, we will use a
numerical example of \cite[Example 1]{8a} for the sake of consistent
comparison.

\begin{Example}
Let the function $T:[0,4]\rightarrow \lbrack 0,4]$ defined by $T(x)=(x+2)^{%
\frac{1}{3}}.$ It is easy to see that $T$ is a contraction mapping. Hence $T$
has a unique fixed point.

In the following table, comparison of the convergence of our new "$K$
iteration process" with the Picard-S iteration, the Thakur New iteration and
the Vatan Two-step iteration processes are given, where $x_{0}=u_{0}=1.99$, $%
\alpha _{n}=\beta _{n}=\frac{1}{4}$ and $n=\overline{1,12}.$%

\begin{table}[h!]
\caption{Iterative values of K, Vatan Two-step, Thakur New and Picard-S iteration processes for $\alpha _{n}=\beta
_{n}=\frac{1}{4},$ for all $n$ and mapping $%
T(x)=(x+2)^{\frac{1}{3}}$.}
      \begin{tabular}{ccccc}
        \hline
         & K & Vatan Two-step & Thakur New & Picard-S\\ \hline
        $x_{0}$ & $1.99$ & $1.99$ & $1.99$ & $1.99$ \\ 
$x_{1}$ & $1.522643193061496$ & $1.527152378405542$ & $1.530163443560674$ & $%
1.530160376515624$ \\ 
$x_{2}$ & $1.521383278248461$ & $1.521453635507796$ & $1.521551978236029$ & $%
1.521551916843118$ \\ 
$x_{3}$ & $1.521379716901169$ & $1.521380654057891$ & $1.521383088492668$ & $%
1.521383087287047$ \\ 
$x_{4}$ & $1.521379706833111$ & $1.521379718941864$ & $1.521379773188262$ & $%
1.521379773164595$ \\ 
$x_{5}$ & $1.521379706804648$ & $1.521379706960085$ & $1.521379708107703$ & $%
1.521379708107238$ \\ 
$x_{6}$ & $1.521379706804568$ & $1.521379706806560$ & $1.521379706830149$ & $%
1.521379706830139$ \\ 
$x_{7}$ & $1.521379706804568$ & $1.521379706804593$ & $1.521379706805070$ & $%
1.521379706805069$ \\ 
$x_{8}$ & $1.521379706804568$ & $1.521379706804568$ & $1.521379706804577$ & $%
1.521379706804577$ \\ 
$x_{9}$ & $1.521379706804568$ & $1.521379706804568$ & $1.521379706804568$ & $%
1.521379706804568$ \\ 
$x_{10}$ & $1.521379706804568$ & $1.521379706804568$ & $1.521379706804568$ & $%
1.521379706804568$ \\ 
$x_{11}$ & $1.521379706804568$ & $1.521379706804568$ & $1.521379706804568$ & $%
1.521379706804568$ \\  \hline
      \end{tabular}
\end{table}

We can easily see that the new $K$ iterations was the first converging one
than the Picard-S, the Thakur New iteration and the Vatan Two-step
iterations.

Graphic representation is given in the following Figure $1,$

\begin{figure}[!h]
\includegraphics[width=10cm,height=8cm]{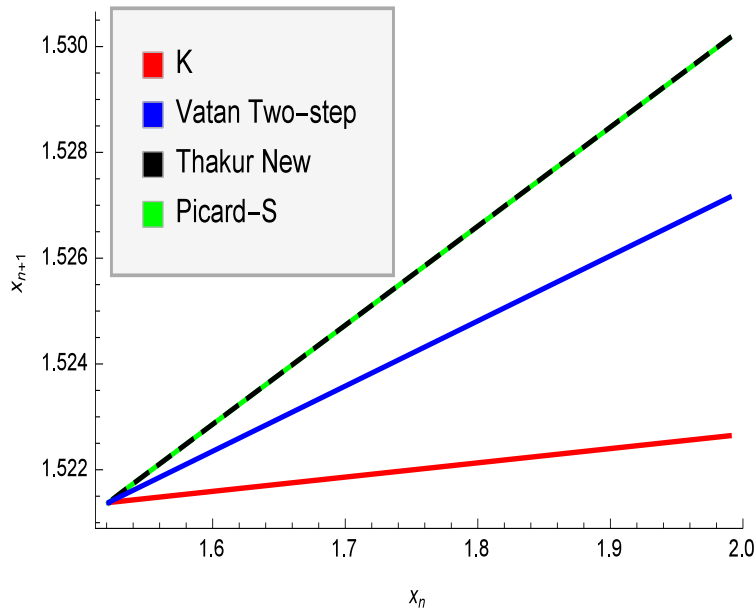} 
\caption{Convergence of K, Vatan Two-step, Thakur New and Picard-S iterations to the fixed point $1.521379706804568$ of mapping $%
T(x)=(x+2)^{\frac{1}{3}}.$}
\end{figure}

\end{Example}

For numerical interpretations first we construct an example of suzuki
generalized nonexpansive mapping which is not nonexpansive.



\newpage


\begin{thebibliography}{99}

\bibitem{1} M. Abbas, T. Nazir, A new faster iteration process applied to
constrained minimization and feasibility problems, Mat. Vesn. 66 (2) (2014)
223--234.

\bibitem{4} R.P. Agarwal, D. O'Regan, D.R. Sahu, Iterative construction of
fixed points of nearly asymptotically nonexpansive mappings, J. Nonlinear
Convex Anal. 8 (1) (2007) 61--79.

\bibitem{AKH} A. Alotaibi, V. Kumar and N. Hussain, Convergence comparison
and stability of Jungck-Kirk type algorithms for common fixed point
problems, Fixed Point Theory and Applications, 2013, 2013:173.

\bibitem{4aa} V. Berinde, Iterative Approximation of Fixed Points, Springer,
Berlin (2007).

\bibitem{4a} R. Chugh, V. Kumar \& S. Kumar, Strong Convergence of a new
three step iterative scheme in Banach spaces. American Journal of
Computational Mathematics 2 (2012) 345--357.

\bibitem{7} K. Goebel, W.A. Kirk, Topic in Metric Fixed Point Theory,
Cambridge University Press, 1990.

\bibitem{7a} F. Gursoy, V. Karakaya, A Picard-S hybrid type iteration method
for solving a differential equation with retarded argument,
arXiv:1403.2546v2, 2014, pp.16.

\bibitem{7b} A.M. Harder, Fixed point theory and stability results for fixed
point iteration procedures, Ph.D. Thesis, University of Missouri-Rolla,
Missouri, 1987.

\bibitem{HKK} N. Hussain, V. Kumar, M.A. Kutbi, On the rate of convergence
of Jungck-type iterative schemes, Abstract Appl. Anal. 2013 (2013) Article
ID 132626.

\bibitem{HJ} N. Hussain et al., On the rate of convergence of Kirk type
iterative schemes, J. Applied Math., Volume 2012, Article ID 526503, 22 pp.

\bibitem{8} S. Ishikawa, Fixed points by a new iteration method, Proc. Am.
Math. Soc. 44 (1974) 147--150.

\bibitem{8aa} I. Karahan \& M. Ozdemir, A general iterative method for
approximation of fixed points and their applications. Advances in Fixed
Point Theory 3(3) (2013).

\bibitem{8a} V. Karakaya, N.E.H. Bouzara, K. Dogan, and Y. Atalan, On
different results for a new two-step iteration method under weak-contraction
mapping in Banach spaces, arXiv:1507.00200v1, 2015, pp.10.

\bibitem{8b} V. Karakaya, F. Gursoy \& M. Erturk, Comparison of the speed of
convergence among various iterative schemes, arXiv preprint arXiv:1402.6080
(2014).

\bibitem{8c} S.H. Khan, A Picard-Mann hybrid iterative process, Fixed Point
Theory Appl. 2013, Article ID 69 (2013).

\bibitem{8d} A.R. Khan, V. Kumar, N. Hussain, Analytical and numerical
treatment of Jungck-Type iterative schemes, Applied Mathematics and
Computation, 231(2014), 521--535.

\bibitem{10} W.R. Mann, Mean value methods in iteration, Proc. Am. Math.
Soc. 4 (1953) 506--510.

\bibitem{11} M.A. Noor, New approximation schemes for general variational
inequalities, J. Math. Anal. Appl. 251 (1) (2000) 217--229.

\bibitem{12} Z. Opial, Weak convergence of the sequence of successive
approximations for nonexpansive mappings, Bull. Am. Math. Soc. 73 (1967)
595--597.

\bibitem{13} W.V. Petryshyn, Construction of fixed points of demicompact
mappings in Hilbert space, J. Math. Anal. Appl. 14 (1966) 276--284.

\bibitem{14} W. Phuengrattana, Approximating fixed points of
Suzuki-generalized nonexpansive mappings, Nonlinear Anal. Hybrid Syst. 5 (3)
(2011) 583--590.

\bibitem{14a} W. Phuengrattana, S. Suantai, On the rate of convergence of
Mann, Ishikawa, Noor and SP-iterations for continuous functions on an
arbitrary interval. Journal of Computational and Applied Mathematics 235
(2011) 3006-3014.

\bibitem{15} B.E. Rhoades, Some fixed point iteration procedures, Int. J.
Math. Math. Sci. 14 (1) (1991) 1--16.

\bibitem{16} B.E. Rhoades, Fixed point iterations using infinite matrices,
III, Fixed Points, Algorithms and Applications, Academic Press Inc. (1977)
337--347.

\bibitem{16a} D.R. Sahu, A. Petrusel, Strong convergence of iterative
methods by strictly pseudocontractive mappings in Banach spaces.
NonlinearAnalysis: Theory, Methods \&Applications 74(17) (2011) 6012-6023.

\bibitem{17} J. Schu, Weak and strong convergence to fixed points of
asymptotically nonexpansive mappings, Bull. Aust. Math. Soc. 43 (1) (1991)
153--159.

\bibitem{18} H.F. Senter, W.G. Dotson, Approximating fixed points of
nonexpansive mappings, Proc. Am. Math. Soc. 44 (2) (1974) 375--380.

\bibitem{18a} S.M. Soltuz, T. Grosan, Data dependence for Ishikawa iteration
when dealing with contractive like operators, Fixed Point Theory and
Applications 242916 (2008) 1-7.

\bibitem{19} T. Suzuki, Fixed point theorems and convergence theorems for
some generalized nonexpansive mappings, J. Math. Anal. Appl. 340 (2) (2008)
1088--1095.

\bibitem{19a} B.S Thakur, D. Thakur, M. Postolache, A new iterative scheme
for numerical reckoning fixed points of Suzuki's generalized nonexpansive
mappings, App. Math. Comp. 275 (2016) 147--155.

\bibitem{19b} X. Weng, Fixed point iteration for local strictly
pseudocontractive mapping, Proc. Amer. Math. Soc. 113 (1991) 727-731.
\end{thebibliography}
\end{document}